\documentclass[graybox]{svmult}
\usepackage{mathptmx}       
\usepackage{helvet}         
\usepackage{courier}        
\usepackage{type1cm}        
\usepackage{makeidx}         
\usepackage{graphicx}        
\usepackage{multicol}        
\usepackage[bottom]{footmisc}
\usepackage{amssymb}
\usepackage{amsmath}

\newcommand{\comment}[1]{}

\newcommand{\pagebrk}{\pagebreak[4]} 
\newcommand{\C}{{\mathbb C}}
\newcommand{\R}{{\mathbb R}}

\newcommand{\Q}{{\mathbb Q}}
\newcommand{\T}{{\mathbb T}}
\newcommand{\Z}{{\mathbb Z}}
\newcommand{\Prob}{{\mathbb P}}
\newcommand{\Esp}{{\mathbb E}}

\newfont{\cmbsy}{cmbsy10}
\newfont{\cmmib}{cmmib10}
\newcommand{\Orden}{\mathop{\hbox{\cmbsy O}}}

\newcommand{\FGauss}{\,_{2}\!\,F_1}

\newcommand{\raisecomma}{\raisebox{2pt}{$,$}}
\newcommand{\raisedot}{\raisebox{2pt}{$.$}}

\def\indistHIGH{\,{\buildrel d \over \rightarrow}\,}
\def\indist{\,{\indistHIGH}\,}

\makeindex             
\begin{document}
\title*{On the sign of the real part of the Riemann zeta-function}
\titlerunning{Sign of the real part of the Riemann zeta-function}
\author{Juan Arias de Reyna, Richard P. Brent and Jan van de Lune}
\institute{Juan Arias de Reyna \at
Universidad de Sevilla, Facultad de Matem\'aticas,
Apdo.1160,  41080-Sevilla, Spain.
\email{arias@us.es}
\and Richard P. Brent \at 
Mathematical Sciences Institute,
Australian National University, Canberra, ACT 0200, Australia.
\email{alf@rpbrent.com}
\and Jan van de Lune \at
Langebuorren 49, 9074 CH Hallum, The Netherlands.
\email{j.vandelune@hccnet.nl}
}
\maketitle

\vspace*{-60pt}
In fond memory of Alfred Jacobus (Alf) van der Poorten 1942--2010\\[30pt]

\abstract*{
We consider the distribution of $\arg\zeta(\sigma+it)$ on fixed lines
$\sigma > \frac12$, and in particular the density 
\[d(\sigma) = \lim_{T \rightarrow +\infty}
        \frac{1}{2T}
        |\{t \in [-T,+T]: |\arg\zeta(\sigma+it)| > \pi/2\}|\,,\]
and the closely related density
\[d_{-}(\sigma) = \lim_{T \rightarrow +\infty}
        \frac{1}{2T}
        |\{t \in [-T,+T]: \Re\zeta(\sigma+it) < 0\}|\,.\]
Using classical results of Bohr and Jessen, we obtain an explicit
expression for the characteristic function $\psi_\sigma(x)$
associated with $\arg\zeta(\sigma+it)$. We give explicit expressions for
$d(\sigma)$ and $d_{-}(\sigma)$ in terms of $\psi_\sigma(x)$.
Finally, we give a practical algorithm for evaluating
these expressions to obtain accurate numerical values of
$d(\sigma)$ and $d_{-}(\sigma)$. 
}

\abstract{
We consider the distribution of $\arg\zeta(\sigma+it)$ on fixed lines
$\sigma > \frac12$, and in particular the density 
\[d(\sigma) = \lim_{T \rightarrow +\infty}
        \frac{1}{2T}
        |\{t \in [-T,+T]: |\arg\zeta(\sigma+it)| > \pi/2\}|\,,\]
and the closely related density
\[d_{-}(\sigma) = \lim_{T \rightarrow +\infty}
        \frac{1}{2T}
        |\{t \in [-T,+T]: \Re\zeta(\sigma+it) < 0\}|\,.\]
Using classical results of Bohr and Jessen, we obtain an explicit
expression for the characteristic function $\psi_\sigma(x)$
associated with $\arg\zeta(\sigma+it)$. We give explicit expressions for
$d(\sigma)$ and $d_{-}(\sigma)$ in terms of $\psi_\sigma(x)$.
Finally, we give a practical algorithm for evaluating
these expressions to obtain accurate numerical values of
$d(\sigma)$ and $d_{-}(\sigma)$. 
}

\section{Introduction} \label{sec:intro}

Several authors, including 
Edwards~\cite[pg.~121]{Edwards},
Gram~\cite[pg.~304]{Gram1},
Hutchinson~\cite[pg.~58]{Hutchinson},
and Milioto~\cite[\S2]{Milioto},
have observed that the real part
$\Re\zeta(s)$ of the Riemann zeta-function $\zeta(s)$ is ``usually positive''.
This is plausible because the Dirichlet series
$\zeta(s) = 1 + 2^{-s} + 3^{-s} + \cdots$
starts with a positive term, and the other terms
$n^{-s}$ may have positive or negative real part. 
In this paper our aim is to make precise the statement that
$\Re\zeta(s)$ is ``usually positive'' for $\sigma := \Re(s) > \frac12$.

Kalpokas and Steuding~\cite{KS}, assuming the Riemann hypothesis, 
have given a sense in which the statement is also true on the critical line
$\sigma = \frac12$. They showed that the mean value of
the set of real values of $\zeta(\frac12+it)$ exists and is equal to~$1$.

We do not assume the Riemann hypothesis, and our results do not 
appear to imply anything about the existence or non-existence of
zeros of $\zeta(s)$ for $\sigma > \frac12$.

Our results depend on the classical results of 
Bohr and Jessen~\cite{BJ1,BJ2} concerning the
value-distribution of $\zeta(s)$ in the half-plane $\sigma > \frac12$.
Since Bohr and Jessen there have been many further results on the value
distribution of various classes of L-functions.  See, for example,
Joyner~\cite{Joyner}, 
Lamzouri~\cite{Lamzouri0,Lamzouri1,Lamzouri2}, 
Laurin\v{c}ikas~\cite{Laurincikas},
Steuding~\cite{Steuding}, and
Voronin~\cite{Voronin2}. However, for our purposes the results of Bohr and
Jessen are sufficient.

After defining our notation, 
we summarise the relevant results of Bohr and
Jessen in \S\ref{sec:BJ}. 
The densities $d(\sigma)$ and $d_{-}(\sigma)$,
defined in \S\ref{sec:arg}, can be expressed in
terms of the characteristic function $\psi_\sigma(x)$ of a certain random
variable $\Im S$ associated with $\arg\zeta(\sigma+it)$.
We consider $\psi_\sigma$ and a related function $I(b,x)$ in
\S\ref{sec:psi}--\S\ref{sec:asympt-I}.
In Theorem~\ref{thm1} 
we use the results of Bohr and Jessen to obtain an
explicit expression 
for $\psi_\sigma(x)$. 
Theorem~\ref{thm2} relates $\log I(b,x)$ to certain polynomials
$Q_n(x)$ which have non-negative integer coefficients with
interesting congruence properties,
and Theorem~\ref{thm:Ibx} gives an asymptotic expansion of $I(b,x)$
which shows a connection between $I(b,x)$ and the Bessel
function $J_0$. Theorem~\ref{thm:psi_bound} 
shows that $\psi_\sigma(x)$ decays
rapidly as $x \to \infty$.

The explicit expression for $\psi_\sigma$ is an infinite product over
the primes, and converges rather slowly. In \S\ref{sec:psi_algorithm} we
show how the convergence can be accelerated to give a practical algorithm
for computing $\psi_\sigma(x)$ to high accuracy. 

In \S\ref{sec:dsigma} we show how $d(\sigma)$ and $d_{-}(\sigma)$ can be
computed using $\psi_\sigma(x)$, and give the results of numerical
computations in \S\ref{sec:numerics}. Finally, in \S\ref{sec:conclusion} we
comment on how our results might be generalised.

Elliott~\cite{Elliott1} determined the characteristic function
$\Psi_\sigma(x)$ of a limiting distribution associated with a certain sequence 
of $L$-functions. We note that Elliott's $\Psi_\sigma(x)$ is the same
function as our $\psi_\sigma(x)$.  For a possible explanation of this
coincidence, using the concept of \emph{analytic conductor}, we refer
to~\cite[Ch.~5]{IK}. Here we merely note that Elliott's method of proof is
quite different from our proof of Theorem~\ref{thm1}, and applies only to
sequences of $L$-functions $L(s,\chi)$ for which
$\chi$ is a non-principal Dirichlet character.

\subsubsection*{Notation} 

$\Z$, $\Q$, $\R$, and $\C$ denote respectively the integers, rationals, reals
and complex numbers.
The real part of $z \in \C$ is denoted by $\Re z$, 
and the imaginary part by $\Im z$.

When considering $\zeta(s)$ we always have $\sigma := \Re\, s$.
Unless otherwise specified, $\sigma > \frac12$ is fixed.

Consider the open set
$G$ equal to $\C$ with cuts along 
$(-\infty +i\gamma,\beta+i\gamma]$ for each
zero or pole $\beta+i\gamma$ of $\zeta(s)$ with $\beta\ge\frac12$. 
Since $\zeta(s)$ is holomorphic and does not vanish on $G$, we may define
$\log\zeta(s)$ on $G$. We take the branch such that 
$\log\zeta(s)$ is real and positive on $(1,+\infty)$.
On $G$ we define $\arg\zeta(s)$ by
\[
\log\zeta(s) = \log|\zeta(s)| + i\cdot\arg\zeta(s).
\label{eq:arg}
\]

\pagebreak[3]
$P$ is the set of primes, and $p \in P$ is a prime.
When considering a fixed prime $p$ we often use the abbreviations 
$b := p^\sigma$ and $\beta := \arcsin(1/b)$.

$|B|$ or $\lambda(B)$ denotes the Lebesgue measure of a set $B\subset\C$
(or $B\subset\R$).
A set $B\subset\C$ is said to be \emph{Jordan-measurable} if 
$\lambda(\partial B) = 0$, where $\partial B$ is the boundary 
of~$B$.\footnote{A bounded set $B$ is 
Jordan-measurable if and only if for each $\varepsilon>0$ we can find two finite
unions of rectangles with sides parallel to the real and imaginary axes, say
$S$ and $T$, such that $S\subseteq B\subset T$ and
$\lambda(T\smallsetminus S)<\varepsilon$
(see for example 
Halmos~\cite{H}).
}

$\FGauss(a,b;c;z)$ denotes the hypergeometric function
of Gauss, see~\cite{AS,Daalhuis}.

\section{Classical results of Bohr and Jessen} \label{sec:BJ}

In \cite{BJ1,BJ2} Bohr and Jessen study several problems regarding the value 
distribution of the zeta function. In particular, for $\sigma>\frac12$
and a given subset $B \subset \C$,
they consider the limit
\[
\lim_{T\to\infty}
\frac{1}{2T}
{|\{t\in\R: |t| < T,\, \log\zeta(\sigma+it)\in B\}|}.
\]
They prove that the limit exists when $B$ is a rectangle with sides parallel 
to the real and imaginary axes. 

Bohr and Jessen also characterize the limit. In modern terminology, they prove 
\cite[Erster Hauptsatz, pg.~3]{BJ2}
the existence of a probability measure $\Prob_\sigma$, absolutely continuous
with respect to Lebesgue measure, such that for any rectangle $B$ as above the
limit is equal to $\Prob_\sigma(B)$.

Finally, they give a description of the measure $\Prob_\sigma$. To express
it in modern language, consider the unit circle $\T=\{z\in\C: |z|=1\}$ with
the usual probability measure $\mu$ (that is $\frac{1}{2\pi}\,d\theta$ if we
identify $\T$ with the interval $[0,2\pi)$ in the usual way).  Let $P$ be
the set of prime numbers. We may consider
$\Omega:=\T^P$ as a probability space with the product measure $\Prob=\mu^P$.
Each point of $\Omega$ is a sequence $\omega=(z_p)_{p\in P}$, with each
$z_p\in\T$. Thus 
$z_p$ may be considered as a random variable. The random variables $z_p$ 
are independent and uniformly distributed on the unit circle. 

\begin{proposition}
Let $\sigma>\frac12$ and for each prime number $q$ let $z_q$ be the 
random variable defined on $\Omega$ such that $z_q(\omega)=z_q$ when 
$\omega=(z_p)_{p\in P}$. The sum of random variables
\[
S:=-\sum_{p\in P}\log(1-p^{-\sigma} z_p) 
  = \sum_{p\in P}\sum_{k=1}^\infty \frac{1}{k}
p^{-k\sigma} z_p^k
\]
converges almost everywhere, so $S$ is a well defined random variable.
\end{proposition}

\begin{proof}
The random variables $Y_p:=-\log(1-p^{-\sigma}z_p)$ are independent. 
The mean value of each $Y_p$ is zero since
\[
\Esp(Y_p)=\frac{1}{2\pi}\int_0^{2\pi} \sum_{k=1}^\infty \frac{1}{k}
p^{-k\sigma} e^{ik\theta}\,d\theta=0.
\]
It can be shown in a similar way that 
$E(|Y_p|^2) \sim p^{-2\sigma}$. 
Thus $\sum_p E(|Y_p|^2)$ converges.
A classical result of probability theory~\cite[Thm.~B, Ch.~IX]{H}
proves the convergence almost everywhere 
of the series for $S$.
\hspace*{\fill}\qed
\end{proof}

\pagebreak[3]
The measure $\Prob_\sigma$ of Bohr and Jessen is the distribution of 
the random variable $S$.
For each Borel set $B\subset \C$, we have 
\[
\Prob_\sigma(B)=\Prob\{\omega\in\Omega: S(\omega)\in B\}.
\]

The main result of Bohr and Jessen is that, for each
rectangle $R$ with sides parallel to the axes,
\begin{equation}\label{eq:main}
\Prob_\sigma(R)=\lim_{T\to\infty}
\frac{1}{2T}
{|\{t\in\R: |t| <T,\, \log\zeta(\sigma+it)\in R\}|}
\end{equation}
and the limit exists.
It is easy to deduce that \eqref{eq:main} is also true for 
each Jordan-measurable subset $R\subset \C$,
and for sets $R$ of the form $\R \times B$, where
$B$ is a Jordan-measurable subset of $\R$.

\section{Some quantities related to the argument of the zeta function}
\label{sec:arg}

Define a measure $\mu_\sigma$ on the Borel sets of $\R$ by 
$\mu_\sigma(B):=\Prob_\sigma(\R\times B)$. 
If we take a Jordan subset $B\subset\R$, 
the main result of Bohr and Jessen 
implies that 
\[
\mu_\sigma(B)=\lim_{T\to\infty}
\frac{1}{2T}
|\{t\in\R: |t| <T,\, \arg\zeta(\sigma+it)\in B\}|.
\]
The measure $\mu_\sigma$ is the distribution function of the random
variable $\Im S$. In fact
\begin{multline*}
\mu_\sigma(B)=\Prob_\sigma(\R\times B)=
\Prob\{\omega\in\Omega : S(\omega)\in\R\times B\}
=\Prob\{\omega\in\Omega : \Im S(\omega)\in B\}.
\end{multline*}

We are interested in the functions
$d(\sigma)$, $d_{+}(\sigma)$, and $d_{-}(\sigma)$ defined by
\[
d(\sigma) :=
\lim_{T\to\infty}
\frac{1}{2T}
{|\{t\in\R: |t|<T,\, |\arg\zeta(\sigma+it)|>\pi/2\}|},
\]
\[
d_+(\sigma) :=
\lim_{T\to\infty}
\frac{1}{2T}
{|\{t\in\R: |t|<T,\, \Re\zeta(\sigma+it)>0\}|},
\]
\[
d_-(\sigma) :=
\lim_{T\to\infty}
\frac{1}{2T}
{|\{t\in\R: |t|<T,\, \Re\zeta(\sigma+it)<0\}|}.
\]
Informally, $d_{+}(\sigma)$ is the probability that $\Re\zeta(\sigma+it)$
is positive; $d_{-}(\sigma)$ is the\linebreak 
probability that
$\Re\zeta(\sigma+it)$ is negative. We show in \S\ref{sec:numerics}
that $d(\sigma)$ is usually a\linebreak 
good approximation to $d_{-}(\sigma)$.
Observe that 
$d(\sigma)=1-\mu_\sigma([-\pi/2,\pi/2])$,\linebreak 
$\;d_+(\sigma)+d_-(\sigma)=1$, and
$d_+(\sigma) = \sum_{k \in \Z}\mu_\sigma(2k\pi-\pi/2,2k\pi+\pi/2)$.

\section{The characteristic function $\psi_\sigma$} \label{sec:psi}

Recall that the {\em characteristic function} $\psi(x)$ of a random variable
$Y$ is defined by the Fourier transform
$\psi(x) := {\rm E}[\exp(ixY)]$.
We omit a factor $2\pi$ in the exponent
to agree with the statistical literature.

\begin{proposition}\label{prop1}
The characteristic function of the random variable $\Im S$ is given by 
\begin{equation}\label{eq:psi_product}
\psi_\sigma(x)=\prod_pI(p^{\sigma}, x),
\end{equation}
where, writing $b := p^\sigma$, $I(b,x)$ is defined by
\begin{equation}\label{defI}
I(b,x) :=
\frac{1}{2\pi}\int_0^{2\pi}\exp\left({-ix \arg(1-b^{-1}e^{i\theta})}\right)\,
d\theta.
\end{equation}
\end{proposition}

\begin{proof}
By definition
\[
\psi_\sigma(x)=\int_{\Omega}\exp\left({i x \Im S(\omega)}\right)\,d\omega=
\int_{\Omega}\prod_p \exp\left({-ix \arg(1-p^{-\sigma}z_p)}\right)\,d\omega.
\]
By independence the integral of the product is the product of the integrals,
so
\[
\psi_\sigma(x)=\prod_p\int_{\Omega} 
\exp\left({-ix \arg(1-p^{-\sigma}z_p)}\right)\,d\omega.
\]
Each random variable $z_p$ is distributed as $e^{i\theta}$ 
on the unit circle, so
\[
\psi_\sigma(x)=\prod_p\frac{1}{2\pi}\int_0^{2\pi}
\exp\left({-ix \arg(1-p^{-\sigma}e^{i\theta})}\right)\,d\theta
=\prod_pI(p^{\sigma},x).
\]
\vspace*{-35pt}

\hspace*{\fill}\qed
\vspace*{10pt}
\end{proof}

\section{The function $I(b,x)$} \label{sec:I}

In this section we study the function $I(b,x)$ defined by~\eqref{defI}.
It is easy to see from~\eqref{defI} that $I(b,x)$ is an even function of
$x$. Hence, from~\eqref{eq:psi_product}, the same is true for $\psi_\sigma(x)$.

\pagebrk
\begin{samepage}	
\begin{proposition}\label{Firstformulas}
Let $b>1$ and $\beta=\arcsin(b^{-1})$.
Then
\[
\begin{split}
I(b,x)&=
\frac{1}{\pi}
\int_{0}^{\pi}\cos\Bigl(x\arctan\frac{\sin t}{b-\cos t}\Bigr)\,dt\\
&=\frac{2b}{\pi}\int_{0}^{\beta}
\frac{\cos(xt)\cos t}{\sqrt{1-b^2\sin^2t}}\,dt
=\frac{2}{\pi}\int_{0}^{1}\cos\Bigl(x \arcsin\frac{t}{b}\Bigr)
\frac{dt}{\sqrt{1-t^2}}\,\raisedot
\end{split}
\]
\end{proposition}
\end{samepage}

\begin{proof}
By elementary trigonometry we find  
\begin{equation}\label{trig}
\arg(1-b^{-1}e^{it})= -\arctan\frac{\sin t}{b-\cos t}\raisedot
\end{equation}
Substituting in \eqref{defI} gives
\[
I(b,x)=\frac{1}{2\pi}\int_0^{2\pi}
\exp\Bigl(ix\arctan\frac{\sin t}{b-\cos t}\Bigr)\,dt
=\frac{1}{\pi}\int_0^{\pi}\cos\Bigl(x\arctan\frac{\sin t}{b-\cos t}\Bigr)\,dt.
\]

To obtain the second representation, note that
$\arctan({\sin t}/({b-\cos t}))$ is increasing on 
the interval $[0,\gamma]$ and decreasing on $[\gamma,\pi]$,
where $\gamma=\arccos b^{-1}$.
We split the integral on $[0,\pi]$ into integrals on 
$[0,\gamma]$ and  $[\gamma,\pi]$. In each of the resulting 
integrals we change variables, putting $u:=\arctan({\sin t}/({b-\cos t}))$.
Then
\[
t=\arccos\left(b\sin^2 u\pm\cos u \sqrt{1-b^2\sin^2u}\;)\right)\,,
\]
where the sign is ``$+$'' on the first interval
and ``$-$'' on the second interval. After some simplification,
the second representation follows. The third representation follows
by the change of variables $t \mapsto \arcsin (t/b)$.
\hspace*{\fill}\qed
\end{proof}

\begin{lemma}\label{lemma3}
For $|t|<1$ and all $x\in\C$,
\begin{equation}\label{cosarcsin}
\cos(2x\arcsin t)=
\FGauss(-x,x;{\textstyle\frac12};t^2)=
1+\sum_{n=1}^\infty \frac{(2t)^{2n}}{(2n)!}\prod_{j=0}^{n-1}(j^2-x^2).
\end{equation}
\end{lemma}
\begin{proof}
In~\cite[eqn.~15.1.17]{AS} (also \cite[eqn.~15.4.12]{Daalhuis})
we find the identity
\[
\cos(2az) = \FGauss(-a,a;{\textstyle\frac12};\sin^2 z).
\]
Replacing $a$ by $x$ and $z$ by $\arcsin t$, we get the first half
of~\eqref{cosarcsin}. The second half follows from the definition
of the hypergeometric function.
\hspace*{\fill}\qed
\end{proof}
\begin{remark}
An independent proof uses the fact that $f(t):=\cos(2x\arcsin t)$ satisfies
the differential equation $(1-t^2) f''(t)-t f'(t)+4x^2f(t)=0$,
where primes denote differentiation with respect to $t$.
\end{remark}
\begin{remark}
When $x\in \Z$, the series~\eqref{cosarcsin} reduces to a polynomial.
\end{remark}

\pagebreak[3]
\begin{proposition}\label{prop4}
For $b>1$ we have
\[
I(b,2x)=\FGauss(-x,x;1;b^{-2})
=1+\sum_{n=1}^\infty \frac{1}{b^{2n} n!^2}\prod_{j=0}^{n-1}(j^2-x^2).
\]
\end{proposition}

\begin{proof}
From Proposition \ref{Firstformulas}, we have
\[
I(b,2x)=\frac{2}{\pi}\int_{0}^{1}\cos\Bigl(2x \arcsin\frac{t}{b}\Bigr)\,
\frac{dt}{\sqrt{1-t^2}}\,\raisedot
\]
The expression of $I(b,2x)$ as a sum follows from Lemma \ref{lemma3},
using a well-known integral for the Beta function $B(n+\frac12,\frac12)$:
\[
\frac{2}{\pi}\int_{0}^{1}\frac{t^{2n}\,dt}{\sqrt{1-t^2}}=
\frac{1}{\pi}B\left(n+{\textstyle\frac12},{\textstyle\frac12}\right)=
\frac{(2n)!}{n!^2 2^{2n}}\,\raisedot
\]
The identification of $I(b,2x)$ as $\FGauss(-x,x;1;b^{-2})$ then
follows from the definition of the hypergeometric function $\FGauss$.
\hspace*{\fill}\qed
\end{proof}

\begin{corollary} \label{cor:rational_Ibx}
If $x \in \Z$, $b^2 \in \Q$ and $b>1$, then $I(b,2x) \in \Q$.
\end{corollary}
\begin{proof}
Since $I(b,2x)$ is even, we can assume that $x \ge 0$.
Applying Euler's\linebreak 
{transformation}~\cite[(15.3.4)]{AS} to the hypergeometric 
representation of Proposition~\ref{prop4}, we
obtain $I(b,2x) = (1-b^{-2})^x\FGauss(-x,1-x;1;1/(1-b^2))$,
but the series for\linebreak 
$\FGauss(-x,1-x;1;z)$ terminates, so is rational for $z \in \Q$.
\hspace*{\fill}\qed
\end{proof}

We can now prove our first main result, which gives 
an explicit expression for the
characteristic function $\psi_\sigma$ defined in
\S\ref{sec:BJ}--\S\ref{sec:psi}.
\begin{theorem}\label{thm1}
For $\sigma > \frac12$,
the characteristic function $\psi_\sigma$ of Proposition~\ref{prop1}
is the entire function given by the convergent infinite product
\begin{equation}
\label{eq:psi_prod}
\psi_\sigma(2x)=\prod_p\Bigl(1+\sum_{n=1}^\infty\frac{1}{n!^2}
\prod_{j=0}^{n-1}(j^2-x^2)\cdot \frac{1}{p^{2n\sigma}}\Bigr).
\end{equation}
\end{theorem}

\begin{proof}
The identity~\eqref{eq:psi_prod} follows from 
Proposition~\ref{prop1} and Proposition \ref{prop4}.
Since $\sum p^{-2\sigma}$ converges,
the infinite product~\eqref{eq:psi_prod} 
converges for all $x\in \C$. 
\hspace*{\fill}\qed
\end{proof}

\section{The function $\log I(b,x)$} \label{sec:logI}

The explicit formula for $\psi_\sigma$ given by Theorem~\ref{thm1}
is not suitable for numerical computation because the infinite product
over primes converges too slowly.  In \S\ref{sec:psi_algorithm} we show how
this difficulty can be overcome.  First we need to consider the
function $\log I(b,x)$.

\begin{theorem}\label{thm2}
Suppose that $b > \max(1, |x|)$.
There exist even polynomials $Q_n(x)$ of degree $2n$ with $Q_n(0)=0$ and  
nonnegative integer coefficients $q_{n,k}$ such that 
\begin{equation} \label{eq_I(b,2x)}
\log I(b,2x)=-\sum_{n=1}^\infty\frac{Q_n(x)}{n!^2}\frac{1}{b^{2n}}=
-\sum_{n=1}^\infty \sum_{k=1}^n
\frac{q_{n,k}}{n!^2}\frac{x^{2k}}{b^{2n}}\,\raisedot
\end{equation}
The polynomials $Q_n(x)$ are determined by the recurrence
\begin{equation}\label{eq:Q_rec}
Q_1(x)=x^2,\quad Q_{n+1}(x) = (n!)^2x^2 +
\sum_{j=0}^{n-1}{n \choose j}{n \choose j+1}Q_{j+1}(x)Q_{n-j}(x)\,.
\end{equation}
Also, the polynomials $Q_n(x)$ satisfy
\begin{equation}\label{eq:Q_ineq}
|Q_n(x)| \le n!(n-1)!\,\max(1,|x|)^{2n}.
\end{equation}
\end{theorem}

\pagebreak[3]
\begin{proof}
By Proposition \ref{prop4} there exist even polynomials $P_n$ with $P_n(0)=0$,
such that 
\[
I(b,2x)=1+\sum_{n=1}^\infty \frac{P_n(x)}{n!^2}\frac{1}{b^{2n}}\,\raisedot
\]
It follows that 
\[
\log I(b,2x)=\sum_{k=1}^\infty\frac{(-1)^{k+1}}{k}
\Bigl(\sum_{n=1}^\infty \frac{P_n(x)}{n!^2}\frac{1}{b^{2n}}\Bigr)^k\raisedot
\]
It is clear that expanding the powers gives a series of the
desired form~\eqref{eq_I(b,2x)}. 

To prove the recurrence for the $Q_n$,
we temporarily consider $x$ as fixed and
define  
$f(y):= I(y^{-1/2},2x)$. Then, by \eqref{eq_I(b,2x)},
\begin{equation}\label{eq:expa_f}
\log f(y)=-\sum_{n=1}^\infty \frac{Q_n}{n!^2}y^n\,.
\end{equation}
By Proposition \ref{prop4} we have $f(y)=\FGauss(x, -x; 1; y)$,
so $f(y)$ satisfies the hypergeometric differential equation
\[
y(1-y)f'' + (1-y)f' + x^2 f = 0,
\]
where primes denote differentiation with respect to $y$.
Define $g(y) := f'(y)/f(y)$.
Then it may be verified
\footnote{Usually a Riccati equation is reduced to a
second-order linear differential equation, see for\\ example
Ince~\cite[\S2.15]{Ince1}. We apply the standard argument in
the reverse direction.
}
that $g(y)$ satisfies the Riccati
equation
\begin{equation}\label{eq:Riccati}
y(g' + g^2) + g + \frac{x^2}{1-y} = 0.
\end{equation}
Let
$g(y) = \sum_{n=0}^\infty g_ny^n$,
where the $g_n$ are polynomials in $x$, e.g.\ $g_0 = -x^2$.
Equating coefficients in~\eqref{eq:Riccati}, we get the recurrence
\begin{equation}\label{eq:g_rec}
g_n = - \Bigl(\frac{1}{n+1}\Bigr)
	\Bigl(x^2 + \sum_{j=0}^{n-1} g_j g_{n-1-j}\Bigr), 
		\;\; {\rm for} \;\; n \ge 0.
\end{equation}
Now, from \eqref{eq:expa_f} and the definitions of $f$ and $g$, we have
\[
g(y) = \frac{f'(y)}{f(y)} = \frac{d}{dy}\log f(y)
	= - \frac{d}{dy} \sum_{n=1}^\infty \frac{Q_n(x)}{n!^2}y^n,
\]
so we see that
\begin{equation}\label{eq:gQ}
g_n = - \frac{Q_{n+1}}{n!(n+1)!}\,\raisedot
\end{equation}
Substituting \eqref{eq:gQ} in (\ref{eq:g_rec}) and simplifying, 
we obtain the recurrence~\eqref{eq:Q_rec}.
\end{proof}

From the recurrence~\eqref{eq:Q_rec} 
it is clear that $Q_n(x)$ is an even polynomial of degree
$2n$, such that $Q_n(0)=0$. Writing 
$Q_n(x)=\sum_{k=1}^ n q_{n,k} x^{2k}$,
we see from the recurrence~\eqref{eq:Q_rec} 
that the coefficients $q_{n,k}$ are nonnegative integers.  

In view of~\eqref{eq:gQ},
the inequality~\eqref{eq:Q_ineq} is equivalent to
$|g_n(x)| \le \max(1,|x|)^{2n+2}$,
which may be proved by induction on $n$, using
the recurrence~\eqref{eq:g_rec}.

Finally, in view of~\eqref{eq:Q_ineq},
the series in~\eqref{eq_I(b,2x)} converge for $b > \max(1,|x|)$.
\hspace*{\fill}\qed

\pagebreak[3]
\begin{corollary} \label{cor:nonzero_I}
If $b > 1$, then $I(b,2x)$ is nonzero in the disk $|x| < b$.
\end{corollary}
\begin{proof}
This follows from the convergence of the series for $\log I(b,2x)$.
\hspace*{\fill}\qed
\end{proof}

\begin{proposition} \label{prop:qnk}
The numbers $q_{n,k}$ are determined by 
$q_{n,1}=(n-1)!^2$ for $n\ge 1$,
and 
\begin{equation}\label{eq:q_rec}
 q_{n+1,k}=\sum_{j=0}^{n-1}
\binom{n}{j}\binom{n}{j+1}\sum_{r=\mu}^\nu q_{j+1,r}q_{n-j,k-r}
\end{equation}
for $2\le k\le n+1$,
where $\mu=\max(1,k-n+j)$ and $\nu=\min(j+1,k-1)$.
Also, $q_{n,k}$ is a positive integer for each $n\ge1$ and $1\le k\le n$. 
\end{proposition}

\begin{proof}
The recurrence is obtained by equating coefficients of $x^{2k}$ in
\eqref{eq:Q_rec}. Positivity of the $q_{n,k}$ for $1 \le k \le n$ follows.
\hspace*{\fill}\qed
\end{proof}
\begin{remark}
We may consider the sum over $r$ in~\eqref{eq:q_rec} to be over all
$r \in \Z$ if we define $q_{n,k}=0$ for $k<1$ and $k>n$. The given
values $\mu$ and $\nu$ correspond to the nonzero terms of the resulting
sum.
\end{remark}

\begin{corollary}\label{cor sum}
We have $\sum_{k=1}^n q_{n,k} = n!\, (n-1)!$.
\end{corollary}

\begin{proof}
This is easily obtained if we substitute 
$x=1$ in the recurrence \eqref{eq:Q_rec}. 
\hspace*{\fill}\qed
\end{proof}

\begin{corollary}\label{cor:BesselConn}
We have
\begin{equation}\label{eq:BesselConn}
q_{n,n}=2^{2n}n!(n-1)!\sum_{k=1}^\infty\frac{1}{j_{0,k}^{2n}}
\end{equation}
where $(j_{0,k})$ is the sequence of positive 
zeros of the Bessel function $J_0(z)$.
\end{corollary}

\begin{proof}
Define $q_n:=q_{n,n}$.
With $k=n+1$, the recurrence \eqref{eq:q_rec} gives, for $n \ge 1$,
\[
q_{n+1}=\sum_{j=0}^{n-1}\binom{n}{j}\binom{n}{j+1}q_{j+1}q_{n-j}=
\sum_{j=1}^{n}\binom{n}{j}\binom{n}{j-1}q_{j}q_{n-j+1}.
\]
This recurrence appears in Carlitz~\cite[eqn.~(4)]{C}, where it is shown that
the solution satisfies~\eqref{eq:BesselConn}.
\hspace*{\fill}\qed
\end{proof}

\begin{remark}
The sequence $(q_{n})$ is A002190 in Sloane's 
on-line encyclopedia of\linebreak 
integer sequences (OEIS),
where the generating function $-\log(J_0(2\sqrt{x}))$
is given.
The numbers $q_{n}$ enjoy remarkable
congruence properties. In fact, \eqref{eq:BesselConn} is analogous to
Euler's identity 
$|B_{2n}| = 2(2n)!\sum_{k=1}^\infty (2\pi k)^{-2n}$,
and the numbers $q_{n}$ are analogous to Bernoulli numbers.
We refer to Carlitz~\cite{C} for further discussion.
\end{remark}
\begin{remark}
There are other recurrences giving the polynomials $Q_n$ and the numbers 
$q_{n,k}$. We omit discussion of them here due to space limitations.
\end{remark}

\begin{table}[t]
\begin{center}
\caption{The coefficients $q_{n,k}$.}
\begin{tabular}{||l||r|r|r|r|r|r|r||}
\hline
\hline
$n\smallsetminus k$ & 1 & 2 & 3 & 4 & 5 & 6 & 7\\
\hline
\hline
1 & 1 & &&&&&\\
2 & 1 & 1 & &&&&\\
3 & 4 & 4 & 4 & &&&\\
4 & 36 & 33 & 42 & 33 & & &\\
5 & 576 & 480 & 648 & 720 & 456 &&\\
6 & 14400 & 10960 & 14900 & 18780 & 17900 & 9460 &\\
7 & 518400 & 362880 &  487200 & 648240 & 730800 & 606480 & 274800\\
\hline
\end{tabular}
\end{center}
\label{default}
\end{table}

\section{Bounds and asymptotic expansions} \label{sec:asympt-I}

Since $I(b,x)$ is an even function of $x$, there is no loss of generality
in assuming that $x \ge 0$ when giving bounds or asymptotic results
for $I(b,x)$. This simplifies the statement of the results. 
Similarly remarks apply to $\psi_\sigma(x)$, which is also an even function.

Consider the first representation of $I(b,x)$ in
Proposition~\ref{Firstformulas}.
If $b$ is large, then
\[
\arctan\left(\frac{\sin \theta}{b - \cos\theta}\right) = \frac{\sin \theta}{b} +
\Orden\left(b^{-2}\right)\,.
\]
However, it is well-known~\cite[\S2.2]{Watson} 
that the Bessel function $J_0(x)$ has an integral 
representation 
\begin{equation} \label{J_0_integral}
J_0(x) = \frac{1}{\pi}\int_0^\pi \cos(x\sin\theta)\,d\theta\,.
\end{equation}
Thus, we expect $I(b,x)$ to be approximated in some sense by $J_0(x/b)$.
A more detailed analysis
confirms this (see Proposition~\ref{prop:Ibx_J0_bd} and 
Corollary~\ref{cor:Bessel_approx}).
The connection with Bessel functions makes Corollary~\ref{cor:BesselConn}
less surprising than it first appears.

\begin{proposition} \label{prop:easy_I_bd}
For all $b > 1$ and $x \in \R$, we have $|I(b,x)| \le 1$.
\end{proposition}
\begin{proof}
This follows from the final integral in Proposition~\ref{Firstformulas}.
\hspace*{\fill}\qed
\end{proof}

\begin{lemma} \label{lemma:arcsin_bd}
For $t \in [0,1]$ and $c_1 = \pi/2 - 1 < 0.5708$, we have 
\[0 \le \arcsin(t) - t \le c_1 t^3.\]
\end{lemma}
\begin{proof}
Let $f(t) = (\arcsin(t) - t)/t^3$.
We see from the Taylor series that $f(t)$
is nonnegative and increasing in $[0,1]$.
Thus $\sup_{t \in [0,1]} f(t) = f(1) = \pi/2-1$.
\hspace*{\fill}\qed
\end{proof}
\begin{lemma} \label{lemma:arcsin_bd2}
Suppose $b > 1$, $t \in [0,1]$, and $c_1$ as in Lemma~\ref{lemma:arcsin_bd}.
Then
\[0 \le b\arcsin(t/b) - t \le c_1 t^3/b^2.\]
\end{lemma}
\begin{proof}
Replace $t$ by $t/b$ in Lemma~\ref{lemma:arcsin_bd}, and multiply both sides
of the resulting\linebreak 
inequality by $b$.
\hspace*{\fill}\qed 
\end{proof}
\begin{proposition} \label{prop:Ibx_J0_bd}
Suppose $b > 1$, $x > 0$, and $c_2 = (2 - 4/\pi)/3 < 0.2423$. Then
\[|I(b,x) - J_0(x/b)| \le c_2 x/b^3.\]
\end{proposition}
\begin{proof}
From the last integral of Proposition~\ref{Firstformulas}, we have
\[
I(b,bx) = \frac{2}{\pi}\int_0^1 \cos\left(bx\arcsin\frac{t}{b}\right)
\frac{dt}{\sqrt{1-t^2}}\,\raisedot
\]
Also, from the integral representation~\eqref{J_0_integral} for $J_0$,
we see that
\[J_0(x) = \frac{2}{\pi}\int_0^1 \cos(xt) \frac{dt}{\sqrt{1-t^2}}\,\raisedot\]
Thus, by subtraction,
\begin{equation} \label{eq:IJdiff}
I(b,bx) - J_0(x) = \frac{2}{\pi}\int_0^1 f(b,x,t)
\frac{dt}{\sqrt{1-t^2}}\,\raisecomma
\end{equation}
where $f(b,x,t) = \cos(bx\arcsin({t}/{b})) - \cos(xt)$.
Using $|\cos x- \cos y)| \le |x-y|$, we have
\[|f(b,x,t)| \le |bx\arcsin(t/b) - xt|.\]
Thus, from Lemma~\ref{lemma:arcsin_bd2},
\[|f(b,x,t)| \le c_1 t^3x/b^2.\]
Taking norms in~\eqref{eq:IJdiff} gives
\begin{equation} \label{eq:sin3_int}
|I(b,bx) - J_0(x)| \le \frac{2c_1 x}{\pi b^2}
\int_0^1 \frac{t^3 dt}{\sqrt{1-t^2}}\,\raisedot
\end{equation}
The integral in~\eqref{eq:sin3_int} is easily seen to have the value $2/3$.
Thus, replacing $x$ by $x/b$ in~\eqref{eq:sin3_int} completes the proof.
\hspace*{\fill}\qed
\end{proof}
\begin{corollary} \label{cor:two_term_bound}
If $b > 1$, $x > 0$, $c_2$ as in Proposition~\ref{prop:Ibx_J0_bd}, and 
$c_3 = \sqrt{2/\pi} < 0.7979$, then
\begin{equation} \label{eq:two_term_bound}
|I(b,x)| \le c_2 x/b^3 + c_3 (b/x)^{1/2}.
\end{equation}
\end{corollary}
\begin{proof}
It is known~\cite[9.2.28--9.2.31]{AS} that 
$|J_0(x)| \le \sqrt{2/(\pi x)}$ for real, positive $x$.
Thus, the result follows from Proposition~\ref{prop:Ibx_J0_bd}.
\hspace*{\fill}\qed
\end{proof}
\begin{remark}
The crossover point in Corollary~\ref{cor:two_term_bound} is
for $b \approx x^{3/7}$:
the first term in~\eqref{eq:two_term_bound} dominates if $b \ll x^{3/7}$;
the second term dominates if $b \gg x^{3/7}$.
\end{remark}
\begin{corollary} \label{cor:uniform_Ibound}
If $x > 1$ and $b \ge x^{1/2}$, then
\[|I(b,x)| \le c_3(b/x)^{1/2}(1 + c_5b^{-1/2}),\]
where $c_2$, $c_3$ are as above, and
$c_5 = c_2/c_3 < 0.3037$.
\end{corollary}
\begin{proof}
From Corollary~\ref{cor:two_term_bound} we have
\[|I(b,x)| \le c_3(b/x)^{1/2}(1 + c_5 x^{3/2}/b^{7/2}).\]
The condition $b \ge x^{1/2}$ implies that $x^{3/2}/b^{7/2} \le b^{-1/2}$.
\hspace*{\fill}\qed
\end{proof}

For the remainder of this section we write $\beta := \arcsin(1/b)$.

\begin{proposition}
For $b > 1$ and real positive $x$, we have 
\begin{equation}\label{eq:laplaceIntegral}
I(b,x)=
-\Re\left( \frac{2i e^{ix\beta}}{\pi}\int_0^\infty 
e^{-x u}\frac{\sqrt{b^2-1}\cosh u-i\sinh u}
{\sqrt{1-(\cosh u+i\sqrt{b^2-1}\sinh u)^2}}\,du\right)\,.
\end{equation}
\end{proposition}
\pagebreak[3]
\begin{proof}
From the second integral in Proposition \ref{Firstformulas} 
we get $I(b,x)=\Re J(b,x)$, where
\[
J(b,x)=\frac{2b}{\pi}\int_0^\beta e^{ixt}
\frac{\cos t}{\sqrt{1-b^2\sin^2t}}\,dt\,.
\]
The function $1-b^2\sin^2t$ has zeros at
$t=\pm\beta+k\pi$ with $k\in\Z$ 
and only at these points.  
Also, $\beta = \arcsin(1/b) \in (0, \pi/2)$.
Hence, if $\Omega$ denotes the complex plane $\C$ with
two cuts along the half-lines $(-\infty,-\beta]$ and $[\beta,+\infty)$,
then the function
${(\cos t)}/{\sqrt{1-b^2\sin^2t}}$
is analytic on $\Omega$.  We consider the branch that is real and positive 
in the interval $(0,\beta)$.
We  apply Cauchy's Theorem  to the half strip  $\Im t>0$, $0<\Re t<\beta$,
obtaining
\[
J(b,x)=\frac{i}{\pi}\int_0^\infty e^{-x u}\frac{2b\cosh u}{\sqrt{1+b^2\sinh^2u}}
\,du \;-\;\frac{i e^{ix\beta}}{\pi}\int_0^\infty e^{-x u}\frac{2b\cos(\beta+iu)}
{\sqrt{1-b^2\sin^2(\beta+iu)}}\,du.
\]
The first integral does not contribute to the real part.  Taking
the real part of the second integral
and simplifying gives \eqref{eq:laplaceIntegral}.
\hspace*{\fill}\qed
\end{proof}

In the following theorem we give an asymptotic expansion of $I(b,x)$.
\begin{theorem} \label{thm:Ibx}
For $b>1$ fixed and real $x\to +\infty$, 
there is an asymptotic expansion of $I(b,x)$. 
If $\beta = \arcsin(1/b)$, the first three terms are given by
\begin{eqnarray*}
I(b,x)=&&\frac{2}{\sqrt{2\pi}}\frac{(b^2-1)^{1/4}}{x^{1/2}}\cos(x\beta-\pi/4)
\\&&+
\frac{(b^2+2)}{4\sqrt{2\pi}}\frac{(b^2-1)^{-1/4}}{x^{3/2}}\sin(x\beta-\pi/4)
\\&&-
\frac{(9b^4-28b^2+4)}{64\sqrt{2\pi}}\frac{(b^2-1)^{-3/4}}{x^{5/2}}\cos(x\beta-\pi/4)+
\Orden\left(\frac{1}{x^{7/2}}\right)\,.
\end{eqnarray*}
\end{theorem}

\begin{proof}
We apply the Laplace method and Watson's Lemma \cite[Ch.~3, pg.~71]{O}
to the representation \eqref{eq:laplaceIntegral}.
\hspace*{\fill}\qed
\end{proof}

\pagebreak[3]

\begin{corollary} \label{cor:I_zeros}
For fixed $b > 1$,
the function $I(b,x)$ has infinitely many real zeros.
\end{corollary}
\begin{proof}
This is immediate from the first term of the asymptotic expansion above.
The zeros are near the points 
$\pm \left(\frac{3\pi}{4}+k\pi\right)/\beta$ for $k \in \Z_{\ge 0}$.
\hspace*{\fill}\qed
\end{proof}

\begin{corollary} \label{cor:inf_many_zeros}
For fixed $\sigma > \frac12$, the function $\psi_\sigma(x)$ has
infinitely many real zeros.
\end{corollary}
\begin{proof}
This is immediate from Proposition~\ref{prop1} and Corollary~\ref{cor:I_zeros}.
\hspace*{\fill}\qed
\end{proof}

\begin{corollary} \label{cor:Bessel_approx}
If $b>1$ and $\beta=\arcsin (1/b)$, then for real $x\to +\infty$ we have
\[
I(b,x)=\beta^{1/2}(b^2-1)^{1/4} J_0(\beta x)+\Orden(x^{-3/2}).
\]
\end{corollary}

\begin{proof}
The Bessel function $J_0(x)$ has an asymptotic expansion which gives
\[
J_0(x)=\left(\frac{2}{\pi x}\right)^{1/2}
\left(\cos(x-\pi/4)+\frac{1}{8x}\sin(x-\pi/4)+
\Orden\left(\frac{1}{x^2}\right)\right).
\]
Therefore, from Theorem~\ref{thm:Ibx}, 
the difference $I(b,x) - \beta^{1/2}(b^2-1)^{1/4} J_0(\beta x)$
is of the order indicated. 
\hspace*{\fill}\qed
\end{proof}

Now we give a bound on the function $I(b,x)$ which is sharper than
Corollary~\ref{cor:two_term_bound} in the region $x \gg b^{7/3}$.
\begin{proposition}\label{boundI}
For $b \ge \sqrt{2}$ and real $x \ge 5$, we have
\begin{equation}\label{eq:Ibx_bound}
|I(b,x)|\le1.1512\sqrt{b/x}\,.
\end{equation}
\end{proposition}

\begin{proof}
We consider the representation \eqref{eq:laplaceIntegral}.
Take $A := \sqrt{b^2-1}$ so the condition $b \ge \sqrt{2}$ implies
that $A \ge 1$.
It can be shown that, for $A \ge 1$ and real $u > 0$, the inequality
\[
\Bigl|\frac{A\cosh u-i\sinh u}
{\sqrt{1-(\cosh u+iA\sinh u)^2}}\Bigr|\le \frac{c_4\sqrt{A}}{\min(u^{1/2},1)}
\]
holds. Here, the optimal constant is $c_4 = \sqrt{\coth(2)} < 1.0185$,
attained at $A = u = 1$. 
(We omit details of the proof, which is elementary but tedious.) 
Hence, from~\eqref{eq:laplaceIntegral},
\begin{eqnarray*}
|I(b,x)|&\le& \frac{2c_4}{\pi}(b^2-1)^{1/4}\Bigl\{\int_0^1 u^{-1/2}e^{-xu}\,du+
\int_1^\infty e^{-xu}\,du\Bigr\}\\
&<&\frac{2c_4\sqrt{b}}{\pi} \left(\sqrt{\pi/x}+e^{-x}/x\right).
\end{eqnarray*}
For $x\ge5$ we have $({2c_4}/{\pi})(\sqrt{\pi/x}+e^{-x}/x)\le 
1.1512/\sqrt{x}$.
\hspace*{\fill}\qed
\end{proof}
\begin{remark} \label{remark:improvements}
The constant 1.1512 in~\eqref{eq:Ibx_bound} can be reduced
if we do not ask for uniformity in $b$.
From Theorem~\ref{thm:Ibx}, we have
\[
|I(b,x)| \le c_3(1-b^{-2})^{1/4}(b/x)^{1/2} + \;
\Orden\left(x^{-3/2}\right) \;\;{\rm as}\;\; x \to +\infty,
\]
so the constant can be reduced to 
$c_3 = (2/\pi)^{1/2} < 0.7979$ for all $x \ge x_0(b)$.
\pagebreak[3]

The following conjecture is consistent with our analytic results,
for example Corollary~\ref{cor:two_term_bound} and Theorem~\ref{thm:Ibx},
and with extensive numerical evidence.
\begin{conjecture}
For all $b > 1$ and $x > 0$, we have 
$|I(b,x)| < {\displaystyle \sqrt{\frac{2b}{\pi x}}}\,\raisedot$
\end{conjecture}
\end{remark}

To conclude this section, we give a bound on $\psi_\sigma(x)$.

\begin{theorem} \label{thm:psi_bound}
Let $\sigma > \frac12$ be fixed. Then $|\psi_\sigma(x)| \le 1$
for all $x \in \R$.  
Also, there exists a positive constant $c \ge 0.47$ and $x_0(\sigma)$ such that
\[
|\psi_\sigma(x)|\le \exp\left(-\,\frac{c\,x^{1/\sigma}}{\log (x^{1/\sigma})}\right)
\;\;{\rm for\;all\;real}\;\; x \ge x_0(\sigma).
\]
\end{theorem}
\begin{proof}
The first inequality is immediate from the definition of $\psi_\sigma(x)$
as the characteristic function of a random variable.

To prove the last inequality, it is convenient to write $y := x^{1/\sigma}$.
Let ${\cal P}(y)$ be the set of primes $p$ in the interval $(y^{1/2}, y]$.
We can assume that $\psi_\sigma(x) \ne 0$, because
otherwise the inequality is trivial.
From Proposition~\ref{prop:easy_I_bd} and Corollary~\ref{cor:uniform_Ibound}, 
we have
\[
|\psi_\sigma(x)| \le \prod_{p \in {\cal P}(y)} |I(p^\sigma,x)|
	\le \prod_{p \in {\cal P}(y)} \left(c_3 (p/y)^{\sigma/2}
	(1+c_5 p^{-\sigma/2})\right),
\]
which implies
\[
-\log|\psi_\sigma(x)| \ge 
\sum_{p \in {\cal P}(y)} \left(-\log(c_3) + 
{\textstyle\frac{\sigma}{2}}(\log y - \log p)\right)
  + \Orden(y^{1-\sigma/2})\,.
\]
Using $\log(c_3) < -0.22$ and $\sigma > 1/2$ gives
\begin{equation}\label{eq:psi_ineq}
-\log|\psi_\sigma(x)| \ge (\pi(y) - \pi(y^{1/2}))
  ({\textstyle\frac{\sigma}{2}}\log y + 0.22)
  - {\textstyle\frac{\sigma}{2}}\!\!\sum_{p \in {\cal P}(y)} \log p
    + \Orden(y^{3/4}),
\end{equation}
where, as usual, $\pi(y)$ denotes the number of primes in the interval $[1,y]$.

From standard results on the distribution of primes~\cite{Tenenbaum}, we have
\[
\pi(y) = \frac{y}{\log y} + \frac{y}{\log^2 y} + 
\Orden\left(\frac{y}{\log^3 y}\right)
\;\;{\rm and}\;\;
\sum_{p \in {\cal P}(y)} \log p = y + \Orden\left(\frac{y}{\log^2 y}\right)\,.
\]
Substituting in \eqref{eq:psi_ineq}, we see that the leading terms of 
order $y$ cancel, leaving
\[ 
-\log|\psi_\sigma(x)| \ge ({\textstyle\frac{\sigma}{2}}+0.22)\frac{y}{\log y}
+ \Orden\left(\frac{y}{\log^2 y}\right)\,.
\]
Since $\frac{\sigma}{2}+0.22 > 0.47$, the
Theorem follows, provided $y$ is sufficiently large.
\hspace*{\fill}\qed
\end{proof}
\begin{remark}
We find numerically that, for $\sigma \in (0.5, 1.1)$,
we can take $c = 1$ and $x_0 = 5$ in Theorem~\ref{thm:psi_bound}.
\end{remark}

\section{An algorithm for computing $\psi_\sigma(x)$} \label{sec:psi_algorithm}

\samepage{
There is a well-known technique, going back at least to
Wrench~\cite{Wrench},
for accurately computing certain sums/products over primes.
The idea is to express what we want to compute in terms of the
\emph{prime zeta function} 
\[
P(s) := \sum_p p^{-s}\;\;\; (\Re(s) > 1).
\]
} 
\pagebreak[3]

\noindent The prime zeta function 
can be computed from $\log\zeta(s)$ using M\"obius inversion:
\begin{equation}\label{eq:Ps_Mobius}
P(s) = \sum_{r=1}^\infty \frac{\mu(r)}{r}\log\zeta(rs)\,. 
\end{equation}
In fact, \eqref{eq:Ps_Mobius} gives the analytic continuation of
$P(s)$ in the half-plane $\Re s > 0$ (see Titchmarsh~\cite[\S9.5]{T}),
but we only need to compute $P(s)$ for real $s > 1$.

\pagebreak[3]
To illustrate the technique, temporarily ignore questions of convergence.
{From} Theorem~\ref{thm2},
we have 
\[
\log I(p^\sigma, x) = -\sum_{n=1}^\infty \frac{Q_n(x/2)}{n!^2}p^{-2n\sigma}.
\]
Thus, taking logarithms in~\eqref{eq:psi_product},
\begin{equation}\label{eq:the_idea_but_wrong}
\log\psi_\sigma(x) = - \sum_p\sum_{n=1}^\infty
\frac{Q_n(x/2)}{n!^2}p^{-2n\sigma}
= - \sum_{n=1}^\infty \frac{Q_n(x/2)}{n!^2}P(2n\sigma).
\end{equation}
Unfortunately, this approach fails, because $\psi_\sigma(x)$
has (infinitely many) real zeros~-- see Corollary~\ref{cor:inf_many_zeros}.
In fact, the series~\eqref{eq:the_idea_but_wrong} converges for
$|x| < |x_1(\sigma)|$, where $x_1(\sigma)$ is the zero
of $\psi_\sigma(x)$ closest to the origin, and diverges for 
$|x| > |x_1(\sigma)|$.

Fortunately, a simple modification of the approach avoids this difficulty.
Instead of considering a product over all primes, we consider the product
over sufficiently large primes, say $p > p_0(x,\sigma)$.  
Corollary~\ref{cor:nonzero_I} guarantees that $I(p^\sigma,x)$ has
no zeros in the disk $|x| < 2p^\sigma$. Thus, to evaluate $\psi_\sigma(x)$
for given $\sigma$ and $x$, we should choose $2p_0^\sigma > |x|$,
that is $p_0 > |x/2|^{1/\sigma}$.  In practice, to ensure rapid
convergence, we might choose $p_0$ somewhat larger, say
$p_0 \approx |4x|^{1/\sigma}$.

For the primes $p \le p_0$, we avoid logarithms and 
compute $I(p^\sigma,x)$ directly from
the hypergeometric series of Proposition~\ref{prop4}.

To summarize, the algorithm for computing $\psi_\sigma(x)$
with absolute error $\Orden(\varepsilon)$,  for $x \in \R$,
is as follows.

\pagebreak[3]
\samepage{
\subsubsection*{Algorithm for the characteristic function $\psi_\sigma(x)$}

\begin{enumerate}
\item
${\displaystyle p_0 \leftarrow \lceil |4x|^{1/\sigma} \rceil.}$
\item \label{step2}
${\displaystyle
A \leftarrow \prod_{p \le p_0} \left( 1+\sum_{n=1}^{N}
 \frac{1}{p^{2n\sigma}n!^2}\prod_{j=0}^{n-1}(j^2-(x/2)^2)\right)\,,
}$ 
where $N$ is sufficiently large that the error in truncating the
sum is $\Orden(\varepsilon)$. [Here $A$ is the product over primes $\le p_0$.]
\item \label{step3}
${\displaystyle
B \leftarrow \exp\left(-\sum_{n=1}^{N'} \frac{Q_n(x/2)}{n!^2}\left\{
 P(2n\sigma) - \sum_{p \le p_0} p^{-2n\sigma}\right\}\right)\,,
}$
where $N'$ is sufficiently large that the error in truncating the  
sum is $\Orden(\varepsilon)$, and $Q_n(x/2)$ is evaluated using
the recurrence~\eqref{eq:Q_rec}.
[Here $B$ is the product over primes $> p_0$.]
\item return $A\times B$.
\end{enumerate}
} 

\pagebreak[3]
\subsubsection*{Remarks on the algorithm for $\psi_\sigma(x)$}
\begin{enumerate}
\item
At step~\ref{step3}, $P(2n\sigma)$ can be evaluated
using equation~\eqref{eq:Ps_Mobius}; time can be saved by precomputing the
required values $\zeta(rs)$.
\item
It is assumed that the computation is performed in floating-point arithmetic 
with sufficiently high precision and exponent range~\cite[Ch.~3]{BZ}.
For efficiency the precision should be varied dynamically as required, 
for example, to compensate for cancellation when summing the 
hypergeometric series at step~\ref{step2}, or when computing
the term $\{P(2n\sigma) - \sum_{p \le p_0} p^{-2n\sigma}\}$ 
at step~\ref{step3}.
\item
At step~\ref{step3} an alternative is to evaluate $Q_n(x/2)$ using a
table of coefficients
$q_{n,k}$; these can be computed in advance using 
the recurrence of Proposition~\ref{prop:qnk}. This saves
time (especially if many evaluations of $\psi_\sigma(x)$ at different
points $x$ are required, as is the case when evaluating $d(\sigma)$), 
at the expense of space and the requirement
to estimate $N'$ in advance.
\item
The algorithm runs in polynomial time, in the sense that the number
of bit-operations required to compute $\psi_\sigma(x)$ with
absolute error $O(\varepsilon)$ is bounded by a polynomial
(depending on $\sigma$ and $x$) in $\log(1/\varepsilon)$.
\end{enumerate}

\section{Evaluation of $d(\sigma)$ and $d_{-}(\sigma)$} \label{sec:dsigma}

In this section we show how the densities $d(\sigma)$ and $d_{-}(\sigma)$
of \S\ref{sec:arg} can be expressed in terms of the characteristic
function $\psi_\sigma$.

\samepage{
\begin{proposition}
For $\sigma>1$, the support of the measure $\mu_\sigma$ 
of \S\ref{sec:arg}
is contained in the compact 
interval $[-L(\sigma),\, L(\sigma)]$, where
\[
L(\sigma):=\sum_p\arcsin(p^{-\sigma}).
\]
\end{proposition}
} 

\pagebreak[3]
\begin{proof}
Recall that $\mu_\sigma$ is the distribution of the random variable
$\Im S$ considered in~\S\ref{sec:BJ}. From~\eqref{trig},
$\Im S$ is equal to the sum of terms 
$-\arctan({(\sin t)}/({p^\sigma-\cos t}))$
whose\linebreak 
values are contained in the interval 
$[-\arcsin p^{-\sigma},\arcsin p^{-\sigma}]$. 
Therefore the range of $\Im S$ is contained in the interval 
$[-L(\sigma),\, L(\sigma)]$. 
\hspace*{\fill}\qed
\end{proof}
\begin{remark} It may be shown that the support 
of $\mu_\sigma$ is exactly the interval $[L(\sigma),L(\sigma)]$.
\end{remark}
\begin{remark}
As in van de Lune~\cite{L}, we define $\sigma_0$ to be the (unique)
real root in $(1,+\infty)$ of the equation $L(\sigma) = \pi/2$,
and $\sigma_1$ to be the real root in $(1,+\infty)$ of\linebreak
$L(\sigma) = 3\pi/2$.  These constants are relevant
in~\S\ref{sec:numerics}.
\end{remark}

\comment{
\begin{proposition}
The measure $\mu_\sigma$ has a density $\rho_\sigma(x)$ with respect
to Lebesgue measure.  The function $\rho_\sigma(x)$  and
the characteristic function $\psi_\sigma(x)$ are in $L^1(\R)\cap L^2(\R)$. 
Hence $\rho_\sigma$ is a continuous function.
\end{proposition}

\begin{proof}
This follows from standard arguments.  The details are omitted.
\hspace*{\fill}\qed 
\end{proof}
} 

\pagebrk
\begin{samepage}
\begin{proposition}\label{prop:d-integral}
For $\sigma>\frac12$,
\begin{equation}\label{eq:d_integral}
d(\sigma) = 1 - \frac{2}{\pi}\int_0^\infty 
\psi_\sigma(x)\sin\left(\frac{\pi x}{2}\right)\,\frac{dx}{x}\,.
\end{equation}
\end{proposition}
\end{samepage}
\begin{proof}
Recall from \S\ref{sec:arg} that 
$1 - d(\sigma) = \mu_\sigma([-\pi/2,\pi/2])$.
Since $\psi_\sigma$ is the characteristic function associated to
the distribution $\mu_\sigma$, a standard result\footnote{Attributed
to Paul L\'evy.} 
in probability theory gives
\[
1 - d(\sigma) = \frac{1}{2\pi}\lim_{X \to \infty}
 \int_{-X}^X \frac{\exp(ix\pi/2)-\exp(-ix\pi/2)}{ix}\psi_\sigma(x)\,dx\,.
\]
Since $\psi_\sigma(x)$ is an even function, 
we obtain~\eqref{eq:d_integral}.
\hspace*{\fill}\qed
\end{proof}

To evaluate $d(\sigma)$ numerically from~\eqref{eq:d_integral}, we have
to perform a numerical integration. The following theorem shows that the
integral may be replaced by a rapidly-converging sum if $\sigma > 1$.
\begin{theorem}		\label{thm:d-exact}
Let $\sigma>1$ and $\ell > \max(\pi/2, L(\sigma))$. Then we have
\begin{equation}\label{eq:l_form}
d(\sigma)=1-\frac{\pi}{2\ell}-\frac{2}{\pi}
\sum_{n=1}^\infty \frac{1}{n}\psi_\sigma\Bigl(\frac{\pi n}{\ell}
\Bigr)\sin\Bigl(\frac{n\pi^2}{2\ell}\Bigr)\,.
\end{equation}
\end{theorem}

\begin{proof}
Consider the function $\widetilde{\rho}(x)$ equal to $\rho_\sigma(x)$ in 
the interval
$[-\ell,\ell]$.
Now extend 
$\widetilde{\rho}(x)$ to the real line $\R$, making it periodic 
with period $2\ell$.
Thus
\[
\widetilde{\rho}(x)=\sum_{n\in\Z} f_n 
\exp\left({\frac{\pi i n x}{\ell}}\right),
\quad\text{where}\quad
f_n=\frac{1}{2\ell}\int_{-\ell}^\ell\widetilde{\rho}(x)
\exp\left({-\frac{\pi i n x}{\ell}}\right)\,dx.
\]
Now $\widetilde{\rho}(x)=\rho_\sigma(x)$ for $|x|\le \ell$ and 
$\rho_\sigma(x)=0$ for $|x|>\ell$. 
Therefore
\[
f_n=\frac{1}{2\ell}\int_{-\ell}^\ell\rho_\sigma(x)
\exp\left(\!{-\frac{\pi i n x}{\ell}}\!\right)\,dx
=\frac{1}{2\ell}\int_{\R}\rho_\sigma(x)
\exp\left(\!{-\frac{\pi i n x}{\ell}}\!\right)\,dx=
\frac{1}{2\ell}\psi_\sigma\Bigl(\frac{\pi n}{\ell}\Bigr)\,.
\]

\pagebreak[3]
\noindent Since $\psi_\sigma(x)$ is an even function,
\begin{equation}\label{eq:rhotilde}
\widetilde{\rho}(x)=\frac{1}{2\ell}\sum_{n\in\Z}
\psi_\sigma\Bigl(\frac{\pi n}{\ell}\Bigr)
\exp\left({\frac{\pi i n x}{\ell}}\right)=
\frac{1}{2\ell}+
\frac{1}{\ell}\sum_{n=1}^\infty
\psi_\sigma\Bigl(\frac{\pi n}{\ell}\Bigr)\cos\frac{\pi n x}{\ell}\,\raisedot
\end{equation}
Now
$d(\sigma)=1-\mu_\sigma([-\pi/2,\pi/2])=
1-\int_{-\pi/2}^{\pi/2}\rho_\sigma(t)\,dt$.
Since $\pi/2\le \ell$, we may replace $\rho_\sigma(t)$ by
$\widetilde{\rho}(t)$ in the integral. 
Hence, multiplying the equality~\eqref{eq:rhotilde}
by the characteristic function of $[-\pi/2,\pi/2]$ and integrating, we get
\eqref{eq:l_form}.
\hspace*{\fill}\qed
\end{proof}
\pagebreak[3]
\begin{remark}
The sum in~\eqref{eq:l_form} can be seen as a numerical quadrature to approximate
the integral in~\eqref{eq:d_integral}, taking a Riemann sum with stepsize
$h = \pi/\ell$.  However, we emphasise that~\eqref{eq:l_form} is {\em exact}
under the conditions stated in Theorem~\ref{thm:d-exact}. This is a
consequence of the measure $\mu_\sigma$ having finite support when
$\sigma > 1$.  If $\sigma \in (\frac12, 1]$ then $\mu_\sigma$ no longer
has finite support and~\eqref{eq:l_form} only gives an approximation;
however, this approximation converges rapidly to the exact result
as $\ell \to \infty$,
because $\mu_\sigma$ is well-approximated by measures with finite support.
\end{remark}
\begin{remark}
If we take $m := 4\ell/\pi$ in the Theorem~\ref{thm:d-exact}, we get the
slightly simpler form 
\begin{equation}\label{eq:m_form}
d(\sigma) = 1 - \frac{2}{m} - 
        \frac{2}{\pi}\sum_{n=1}^\infty
\frac{1}{n}\psi_\sigma\left(\frac{4n}{m}\right)\sin\left(\frac{2\pi
n}{m}\right)
\end{equation}
for $m > \max (2, M(\sigma))$, where $M(\sigma) = 4L(\sigma)/\pi$.
A good choice if $L(\sigma) < \pi$ is $m = 4$; 
then only the odd terms in the sum~\eqref{eq:m_form} contribute.
\end{remark}

\subsubsection*{Computation of $d_{-}(\sigma)$}

Recall that $d_{-}(\sigma)$ is the probability that
$\Re\zeta(\sigma+it) < 0$.
Let $a_k = a_k(\sigma)$ be the probability that
$|\arg\zeta(\sigma+it)| > (2k+1)\pi/2$, that is
\[
a_k := 1 - \mu_\sigma([-(k+{\textstyle\frac12})\pi,\,
	(k+{\textstyle\frac12})\pi]).
\]
Then
\begin{equation}\label{eq:dminus}
d_{-}(\sigma) = \sum_{k=0}^\infty (a_{2k}-a_{2k+1})
= \sum_{k=0}^\infty (-1)^k a_k\,.
\end{equation}
We have seen that, for $\sigma > 1$ and $m > \max(2,4L(\sigma)/\pi)$,
eqn.~\eqref{eq:m_form} gives $a_0 = d(\sigma)$.
Similarly, under the same conditions we have
\begin{equation}\label{eq:ak}
a_k = 1 - \frac{4k+2}{m} - \frac{2}{\pi}\sum_{n=1}^\infty
 \frac{1}{n}\psi_\sigma\left(\frac{4n}{m}\right)
  \sin\left(\frac{(4k+2)\pi n}{m}\right).
\end{equation}
Using \eqref{eq:dminus} and \eqref{eq:ak} in conjuction with an
algorithm for the computation of $\psi_\sigma$, we can compute
$d_{-}(\sigma)$ and also, of course, $d_{+}(\sigma) = 1 - d_{-}(\sigma)$.
If $\sigma \in (\frac12, 1]$ then we can take the limit of~\eqref{eq:ak}
as $m \to \infty$, or use an analogue of
Proposition~\ref{prop:d-integral}, to evaluate the constants $a_k$.

\pagebreak[3]
\section{Numerical results} \label{sec:numerics}

In~\cite{rpb246} we described a computation of the first fifty intervals 
($t > 0$) on which $\Re\zeta(1+it)$ takes negative values. The first such
interval occurs for $t \approx 682112.9$, and has length $\approx 0.05$.
From the lengths of the first fifty intervals we estimated that\linebreak
$d_{-}(1) \approx 3.85\times 10^{-7}$. We also mentioned a Monte Carlo
computation which gave $d_{-}(1) \approx 3.80\times 10^{-7}$. 
The correct value is $3.7886\ldots\times 10^{-7}$.
The difficulty of improving the accuracy of these computations or of
extending them to other values of $\sigma$ was one motivation for the
analytic approach of the present paper.

The algorithm of \S\ref{sec:psi_algorithm} was implemented independently by
two of us, using in one case Mathematica and in the other Magma. The
Mathematica implementation precomputes a table of coefficients $q_{n,k}$;
the Magma implementation uses the recurrence for the polynomials $Q_n$
directly. The results obtained by both implementations are in agreement,
and also agree (up to the expected statistical error)
with results obtained by the Monte Carlo method 
in the region $0.6 \le \sigma \le 1.1$ where the latter method 
is feasible. 

Table~2 
gives some computed values of $d(\sigma)$
for $\sigma \in (0.5,1.165]$.  From van de Lune~\cite{L} we know
that $d(\sigma) = d_{-}(\sigma) = 0$ for $\sigma \ge \sigma_0 \approx
1.19234$. Table~2 shows that $d(\sigma)$ is very small for
$\sigma$ close to $\sigma_0$. For example, $d(\sigma) < 10^{-100}$
for $\sigma \ge 1.15$. The small size of $d(\sigma)$
makes the computation difficult for $\sigma \ge 1.15$.
We need to compute
$\psi_\sigma(4n/m)$ to more than 100 decimal places
to compensate for cancellation in the sum~\eqref{eq:m_form}, in order to get
any significant figures in $d(\sigma)$.

\begin{table}[t]
\caption{$d(\sigma)$ for various $\sigma \in (0.5, 1.165]$}
\begin{center}
\begin{tabular}{|l|c|}
\hline
$\;\;\;\;\;\;\;\;\sigma$ & $d(\sigma)$\\
\hline
$0.5+10^{-11}$ & $0.6533592249148917497 \phantom{\times 10^{-00\;\;}}$\\
$0.5+10^{-5}$  & $0.4962734204446697434 \phantom{\times 10^{-00\;\;}}$\\
$0.6    $      & $7.9202919267432753125 \times 10^{-2\;\;}$\\
$0.7    $      & $2.5228782796068962969 \times 10^{-2\;\;}$\\
$0.8    $      & $5.1401888600187247641 \times 10^{-3\;\;}$\\
$0.9    $      & $3.1401743610642112427 \times 10^{-4\;\;}$\\
$1.0    $      & $3.7886623606688718671 \times 10^{-7\;\;}$\\
$1.1    $      & $6.3088749952505014038 \times 10^{-22\;}$\\
$1.15   $      & $1.3815328080907034247 \times 10^{-103}$\\
$1.16   $      & $1.1172074815779368125 \times 10^{-194}$\\
$1.165  $      & $1.2798207752318534603 \times 10^{-283}$\\
\hline
\end{tabular}
\end{center}
\end{table} 

\pagebreak[3]
Selberg~\cite{Selberg1} (see also~\cite{Joyner,Kuhn,Tsang})
showed that, for $t \sim {\rm unif}(T,2T)$,
\begin{equation}\label{eq:Selberg_normal_dist}
\frac{\log\zeta(1/2+it)}{\sqrt{\frac{1}{2}\log\log T}} \indist X+iY
\end{equation}
as $T \rightarrow \infty$, with $X,Y \sim N(0,1)$.
This implies that $d(1/2)=1$, but gives no indication of
the speed of convergence of $d(\sigma)$ as $\sigma \downarrow \frac12$.
Table~2 shows that convergence is
very slow~-- for $\sigma - \frac12 \ge 10^{-11}$ we have
$d(\sigma) < \frac23$. 

It appears from numerical computations that $\Re\zeta(1/2+it)$ is
``usually positive'' for those values of $t$ for which computation is
feasible. This is illustrated by several of the Figures in~\cite{X-ray}.
Because the function $\sqrt{\log\log T}$ grows so slowly,
the region that is feasible for computation may not show the typical
behaviour of $\zeta(\sigma+it)$ for large $t$ on or close to the critical
line $\sigma = \frac12$.

Table~3 
gives the difference $d(\sigma) - d_{-}(\sigma)$.
For $\sigma > 0.8$, there is no appreciable difference between
$d(\sigma)$ and $d_{-}(\sigma)$.  This is because the probability 
that\linebreak
$|\arg\zeta(\sigma+it)| > 3\pi/2$ is very small in this region.
Indeed, $d(\sigma) = d_{-}(\sigma)$ for all 
$\sigma \ge \sigma_1 \approx 1.0068$, 
where $\sigma_1$ is the positive real root of 
$L(\sigma) = 3\pi/2$.

There is an appreciable difference between $d(\sigma)$ and $d_{-}(\sigma)$
very close to the critical line.
For example,
$d_{-}(0.5+10^{-11}) \approx 0.4986058426$,
but $d(0.5+10^{-11}) \approx 0.6533592249$.
Our numerical results suggest that
$\lim_{\sigma \downarrow 1/2}d_{-}(\sigma) = 1/2$.

It is plausible that $d_{-}(\frac12) = d_{+}(\frac12) = \frac12$,
but Selberg's result \eqref{eq:Selberg_normal_dist}
does not seem to be strong enough to imply this.

\begin{table}[t] 
\caption{The difference $d(\sigma) - d_{-}(\sigma)$}
\begin{center}
\begin{tabular}{|c|c|}
\hline
$\sigma$ & $d(\sigma)-d_{-}(\sigma)$\\
\hline
$0.5+10^{-11}$ & $0.1547533823\;\;\;$\\
$0.6$ & $8.073328981 \times 10^{-11}$\\
$0.7$ & $2.676004882 \times 10^{-32}$\\ 
$0.8$ & $7.655052120 \times 10^{-210}$\\
\hline
\end{tabular}
\end{center}
\vspace*{-20pt}
\end{table} 

\section{Conclusion}	\label{sec:conclusion}

We have shown a precise sense in which $\Re\zeta(s)$ is ``usually positive'' 
in the half-plane $\sigma = \Re(s) > \frac12$, 
given an explicit expression for the characteristic function 
$\psi_\sigma$, and given a feasible algorithm for the
accurate computation of $\psi_\sigma$, and consequently for the
computation of the densities $d(\sigma)$ and $d_{-}(\sigma)$.

Our results could be generalised to cover Dirichlet L-functions because
the character $\chi(p)$ in
the Euler product
\[
L(s,\chi) = \prod_p (1 - \chi(p)p^{-s})^{-1}
\]
can be absorbed into the random variable~$z_p$ whenever $|\chi(p)| = 1$.
Thus, it would only be necessary to omit, from sums/products over primes,
all primes $p$ for which $\chi(p)$ is zero, i.e.~the finite number of 
primes that divide the modulus of the L-function.
This would, of course, change the numerical results.
Nevertheless,  we
\hbox{expect} $\Re L(s,\chi)$ to be ``usually positive'' 
for $\Re(s) > \frac12$.

\pagebreak[3]

\end{document}